\documentclass[12pt]{amsart}
\usepackage{amssymb,hyperref} 
\newtheorem{thm}{Theorem}[section]
\newtheorem{cor}[thm]{Corollary}
\newtheorem{lem}[thm]{Lemma}
\newtheorem{prop}[thm]{Proposition}
\theoremstyle{definition}
\newtheorem{defn}[thm]{Definition}
\newtheorem{rem}[thm]{Remark}

\newtheorem{qu}[thm]{Question}

\numberwithin{equation}{section}

\newcommand{\set}[1]{\left\{#1\right\}}

\newcommand{\pr}[1]{\left(#1\right)}
\newcommand{\valu}[1]{\left\langle #1\right\rangle}

      \newcommand{\mf}[1]{\mathbf{#1}}
\begin{document}

\title[Nilpotent probability of compact groups]{Nilpotent probability of compact groups}%

\author[A. Abdollahi]{Alireza Abdollahi}%
\address{Department of Pure Mathematics, Faculty of Mathematics and Statistics, University of Isfahan, Isfahan 81746-73441, Iran.} 
\email{a.abdollahi@math.ui.ac.ir}%
\author[M. Soleimani Malekan]{Meisam Soleimani Malekan}%
 \address{School of Mathematics, Institute for Research in Fundamental Sciences, Tehran, Iran.} 
\email{msmalekan@gmail.com}

\subjclass[2010]{20E18; 20P05}%
\keywords{Nilpotent probability; Compact group}%

\begin{abstract}
Let $k$ be any positive integer and $G$ a compact (Hausdorff) group. Let $\mf{np}_k(G)$ denote the probability that $k+1$ randomly 
chosen elements $x_1,\dots,x_{k+1}$ satisfy $[x_1,x_2,\dots,x_{k+1}]=1$. 
We study the following problem: If $\mf{np}_k(G)>0$ then, does there exist an open nilpotent subgroup of class at most $k$? The answer is positive for profinite groups and we give a new proof. We also prove that the connected component $G^0$ of $G$ is abelian and there exists a closed normal nilpotent subgroup $N$ of class at most $k$ such that $G^0N$ is open in $G$. 
\end{abstract}
\maketitle
\section{Introduction and Results}
Let $G$ be a compact (Hausdorff) group and denote by ${\mf m}_G$ the (unique) Haar measure on $G$. The following question is proposed in \cite[Question 1.20.]{MTVV}:
\begin{qu}\label{Question}
Let $G$ be a compact group, and suppose that
$${\mathcal N}_k(G):=\{(x_1,\dots, x_{k+1}) \in G^{k+1} \;|\; 	[x_1, \dots, x_{k+1}] = 1\},$$
has positive Haar measure in $G^{k+1}$. Does $G$ have an open $k$-step nilpotent
subgroup?
\end{qu}
 Here $[x_1,x_2]=x_1^{-1}x_2^{-1}x_1 x_2$ and  $[x_1,\dots,x_{k+1}]$ is inductively defined by $[[x_1,\dots,x_k],x_{k+1}]$ for all $k>1$. The Haar measure of $\mathcal N_k(G)$ is called the degree of $k$-nilpotence of the group $G$, and denoted in \cite{MTVV} by $\mf{dc}^k(G)$.
 \par Denote by $\mathsf{PA}_k$ the class of compact groups giving positive answer to Question \ref{Question}. 
 In \cite{LP} it is proved that  $\mathsf{PA}_1$ contains profinite groups and in \cite{HR} it is proved that  all compact groups includes $\mathsf{PA}_1$.
 In \cite{AS0} it is shown that $\mathsf{PA}_2$ contains all compact groups.
 Martino et al. have shown that $\mathsf{PA}_k$ contains all profinite groups for any $k$ (see \cite[Theorem 1.19.]{MTVV}). Here in Theorem \ref{anotherproof} below, we give another proof of the latter result without using  submultiplicativity of the $k$-nilpotence degree of a finite group  i.e. for a normal subgroup $N$ of a finite group $G$, $\mf{dc}^k(G)\leq\mf{dc}^k(G/N)\mf{dc}^k(N)$ (see \cite[Theorem 1.21.]{MTVV}).\\ 
 
 Our main result is the following.
\begin{thm}\label{main} Let $G$ be a compact group with $\mf{np}_k(G)>0$. Then the connected component $G^0$ of $G$ is abelian and there exists a closed normal nilpotent subgroup $N$ of class at most $k$ such that $G^0N$ is open in $G$. 
\end{thm}

\section{Relative nilpotent probability and submultiplicativity}
We  denote the set of continuous functions on a compact group $G$ by $\mathcal C(G)$. If $N$ is a closed normal subgroup of $G$, then $G/N$ is a compact group, equipped with quotient topology. There exists a projection-like map $P:\mathcal C(G)\rightarrow\mathcal C(G/N)$ defined by 
\[ Pf(\overline{g}):=\int_Nf(gy)\,\text dy,\]
where $\text dy$ stands for the Haar measure on $N$. We will use the following two simple lemmas repeatedly:
\begin{lem}\label{Gotosubgroups}
Let $G$ be a compact group and $N$ be a closed normal subgroup of $G$. Then, for all $f\in\mathcal C(G)$ 
\[ \int_Gf(g)\,\text dg=\int_{G/N}Pf(\overline{g})\,\text d\overline{g}.\]
\end{lem}
\begin{proof}
The functional 
\[ I:\mathcal C(G)\rightarrow\mathbb C, \quad f\mapsto\int_{G/N}Pf(\overline{g})\,\text d\overline{g}\]
is nonzero, positive and left-invariant, also $I(1)=1$, so by the uniqueness of Haar integral the result holds.
\end{proof}
\begin{rem}
We have an analog of Lemma \ref{Gotosubgroups} whenever $H$ is  a closed not necessarily normal subgroup,  see \cite[(2.49) Theorem]{folland}.
\end{rem}
The Uniqueness of Harr measure has another achievement; consider a continuous epimorphism $\phi:G\rightarrow H$ between compact groups. Then
 \[I:\mathcal C(H)\rightarrow\mathbb C, \quad f\mapsto\int_Gf\circ\phi(g)\,\text dg\]
  defines a nonzero, positive and left-invariant functional. Aslo, $I(1)=1$, so we have 
\begin{lem}\label{themeasureinquoitient}
Let $\phi:G\rightarrow H$ be a continuous epimorphism between compact groups. Then 
\[\int_Hf(h)\,\text dh=\int_Gf\circ\phi(g)\,\text dg\quad(f\in\mathcal C(H)).\]
\end{lem}
\begin{lem}
Let $H$ be a closed subgroup of a group $G$. Then $yC_G(x)\cap H$ is either empty or a coset of $C_H(x)$, for all $x,y\in G$. In particular, we have $\mf m_H(yC_G(x)\cap H)\leq\mf m_H(C_H(x))$, for all $x,y\in G$.
\end{lem}
\begin{proof}
If $z\in C_G(x)\cap y^{-1}H$, then one can easily check that $C_G(x)\cap y^{-1}H=zC_H(x)$.
\end{proof}

\begin{prop}\label{npleqcp}
Let $H$ be a closed subgroup of a compact group $G$. Then $\mf{np}(H;x,y)\leq\mf{cp}(H)$ for all $x$ and $y$ in $G$.
\end{prop}
\begin{proof}
By the above lemma, we have 
\begin{align*}
\mf{np}(H;x,y)&=\int_H\int_H1_{\mathcal N(H;x,y)}(\xi,\eta)\,\text d\eta\,\text d\xi\\
&=\int_H\mf m_H(x^{-1}C_G(y\eta)\cap N)\,\text d\eta\\
&\leq\int_H\mf m_H(C_H(y\eta))\,\text d\eta\\
&=\int_H\mf m_H(y^{-1}C_G(\xi)\cap N)\,\text d\xi\\
&\leq\int_H\mf m_H(C_H(\xi))\,\text d\xi=\mf{cp}(H).
\end{align*}
\end{proof}
\begin{rem}\label{npleqnp}
In \cite[Lemma 6.2.]{MTVV}, using combinatorics methods, the authors have shown that for a normal subgroup $N$ of a \textit{finite} group $G$, 
\[\mf{np}(N; x_1,\dotsc, x_n)\leq\mf{np}(N;\overbrace{1,\dotsc,1}^{n\text{-times}})\]
for all $x_k\in G$, $1\leq k\leq n$. We could not implement these methods in the case of infinite compact groups to have the above inequality for a closed normal subgroup $N$, and of course we do not need this generalization.
\end{rem}
Now, we are going to prove that $\mf{np}(N; x_1,\dotsc, x_n)$ is bounded away from 1, for all closed normal subgroup $N$  of a compact group $G$ which is not nilpotent of class at most $k$ and all $x_k\in G$, $1\leq k\leq n$. 
\begin{prop}\label{2.4n}
Let $H$ be a closed subgroup of a compact group $G$. Then  
\[  \mf{np}(H; x_1,\dotsc,x_{n+1})\leq\frac12\pr{1+ \mf{np}\pr{\frac H{\mf Z\cap H}; \overline{x_1},\dotsc,\overline{x_n}}},\]
for all $x_1,\dotsc,x_n\in G$, where $\mf Z$  is the center of $G$, and $\overline x$ denotes the image of $x$ in $H/(\mf Z\cap H)$.
\end{prop}
\begin{proof}
Put $K:=\mf Z\cap H$, 
\[ X=\set{\mf y=(y_1,\dotsc,y_n)\in H^n: [x_1y_1,\dotsc,x_ny_n]\in K}\]
and 
\[ \mathcal N=\mathcal N\pr{\frac HK; \overline{x_1},\dotsc,\overline{x_n}}.\]
Then one can easily verify that 
\begin{align}\label{1}
1_X(\mf y\mf z)=1_X(\mf y)\quad (\mf y\in H^n, \mf z\in K^n)
\end{align}
($1_X$ is the characteristic function of $X$) and 
\begin{align}\label{2}
1_X(\mf y)=1\quad\text{if and only if}\quad 1_{\mathcal N}(\bar{\mf y})=1.
\end{align}
Now, we have 
\begin{align*}
 \mf{np}(H; x_1,\dotsc,x_{n+1})&=\int_{H^n}\mf m_H\pr{x_{n+1}^{-1}C_G([x_1y_1,\dotsc,x_ny_n])\cap H}\,\text d\mf y\\
 &\leq\int_{H^n}\mf m_H\pr{C_H([x_1y_1,\dotsc,x_ny_n])}\,\text d\mf y\\
&=\int_{X}\mf m_H\pr{C_H([x_1y_1,\dotsc,x_ny_n])}\,\text d\mf y\\
&\quad\,+\int_{H^n\setminus X}\mf m_H\pr{C_H([x_1y_1,\dotsc,x_ny_n])}\,\text d\mf y\\
&\leq\mf m_{H^n}(X)+\frac12(1-\mf m_{H^n}(X))\\
&=\frac12\pr{1+\mf m_{H^n}(X)}.
\end{align*}
The latter inequality holds because on $X^c$, $C_G([x_1y_1,\dots,x_ny_n])$ has index at least 2. We compute $\mf m_{H^n}(X)$: 
\begin{align*}
\mf m_{H^n}(X)&=\int_{H^n}1_{X}(\mf y)\,\text d\mf y\\
&=\int_{H^n/K^n}\int_{K^n}1_{X}(\mf y\mf z)\,\text d\mf z\,\text d\bar{\mf y}
, \quad \text{by Lemma \ref{Gotosubgroups},}
\\
&=\int_{H^n/K^n}1_{X}(\mf y)\,\text d\bar{\mf y}, 
\quad \text{by (\ref{1})},\\
&=\int_{(H/K)^n}1_{\mathcal N}(\bar{\mf y})\,\text d\bar{\mf y},
\quad\text{by Lemma \ref{themeasureinquoitient}},\\
&= \mf{np}\pr{\frac H{\mf Z\cap H}; \overline{x_1},\dotsc,\overline{x_n}}.
\end{align*}
\end{proof}
By Theorem \ref{2.4n} we obtain the following characterization of closed and nilpotent normal subgroups of a compact group:
\begin{prop}\label{nocamn}
Let $H$ be a closed subgroup of a compact group $G$. Then $H$ is nilpotent of class at most $n$ if and only if 
\[\mf{np}(H;x_1,\dotsc,x_{n+1})=1\]
 for some $x_k\in G$, $1\leq k\leq n+1$.
\end{prop}
\begin{proof}
Only one direction needs proof.\par 
Using induction on $n$, we prove the result. For $n=1$, assume that $\mf{np}(H; x,y)=1$ for some $x,y\in G$, so by Proposition \ref{npleqcp}, $\mf{cp}(H)=1$, whence implies that $H$ is an abelian group. 
Assuming that the result is valid for $n-1$, we prove its correctness for $n$. If 
\[\mf{np}(H;x_1,\dotsc,x_{n+1})=1\]
 for some $x_k\in G$, $1\leq k\leq n+1$, then by Proposition \ref{2.4n}, we get 
 $\mf{np}\pr{\frac H{\mf Z\cap H}; \overline{x_1},\dotsc,\overline{x_n}}=1$. Now, by induction hypothesis, $\frac H{\mf Z\cap H}$ is nilpotent of class at most $n-1$, so, $H$ is nilpotent of class at most $n$.
\end{proof}
\begin{defn}\label{rknil}
Let $H$ be a closed subgroup of a compact group $G$. The \textit{the relative $k$-nilpotence of $H$ in $G$}, denoted by $\mf{np}_{k,G}(H)$, defined to be
\[\sup\set{\mf{np}(H, x_1,\dotsc, x_{k+1}): x_1,\dotsc, x_k\in G}.\]
When $H=G$, then we remove $G$ from the index and simply write $\mf{np}_k(G)$. 
\end{defn}
\begin{rem}\label{indepen}
Since the Haar measure is left-invariant, one can easily verify that 
$\mf{np}_k(G)=\mf{np}(G;x_1,\dotsc,x_{k+1})$ for all $x_1,\dotsc,x_{k+1}\in G$. So, $\mf{np}_k(G)$ is nothing but the degree of $k$-step nilpotence of $G$ defined in \cite{MTVV}. Also, by Lemma \ref{npleqcp}, we have $\mf{np}_{1,G}(H)=\mf{cp}(H)$ for a closed subgroup $H$ of a compact group $G$.
\end{rem}
The following result follows inductively from Proposition \ref{nocamn}, \ref{npleqcp} and \cite[Theorem]{Gus}.
\begin{thm}\label{3/2}
	Let $H$ be closed subgroup of a compact group $G$ which is not nilpotent of class at most $n$. Then $\mf{np}_{k,G}(H)\leq 1-\frac{3}{2^{k+1}}$, for all positive integer $k$. 
\end{thm}
By Theorem \ref{3/2} in hand, we give another (simple) proof for the following statement.
\begin{thm}(\cite[Theorem 1.19.]{MTVV})\label{anotherproof}
Let $G$ is a profinite group with $\mf{np}_k(G)>0$. Then $G$ has an open $k$-step nilpotent subgroup.
\end{thm}
\begin{proof}
By \cite[Lemma 2.4.]{AS}, there is an open normal subgroup $N$ along with elements $x_1,\dotsc, x_{k+1}\in G$ such that $\mf{np}(N;x_1,\dotsc, x_{k+1})> 1-\frac{3}{2^{k+1}}$, so by Theorem \ref{3/2}, $N$ should be nilpotent of class at most $k$.
\end{proof}

If $B$ is a measurable subgroup of a compact group $A$, then $\mf m_A(B)=[A:B]^{-1}$, by convention on ``$ \frac1\infty=0 $''. We have 
\begin{lem}\label{A=(A/B)B}
Let $G$ be a compact group, $H$ be a closed subgroup of $G$ and $N$ be a normal closed subgroup of $G$. Then 
\[ \mf m_G(H)=\mf m_N(N\cap H)\mf m_{G/N}\pr{NH/N}.\]
\end{lem}
\begin{proof}
The following is well known
\[ [G:H]=[G:NH][NH:H]=[G/N:NH/N][N:N\cap H].\]
This along with the explanation before the lemma implies the result.
\end{proof}
\begin{thm}\label{increase wrt N}
Let $G$ be a compact group with closed subgroups $H$ and closed normal subgroup $N$ contained in $H$. Then $\mf{np}_{k,G}(H)\leq\mf{np}_{k,G/N}(H/N)\mf{np}_{k,G}(N)$. 
\end{thm}
\begin{proof}
For $x_1,\dotsc,x_{k+1}\in G$ we have
\begin{align*}
\mf{np}(H;x_1,\dotsc,x_{k+1})&\leq\int_{H^k}\mf m_H(C_H([x_1y_1,\dotsc,x_ky_k]))\,\text d\mf y.
\end{align*}
By Lemmas \ref{Gotosubgroups} and \ref{A=(A/B)B},  
\begin{footnotesize}
\begin{align*}
\int_{H^k}&\mf m_H(C_H([x_1y_1,\dotsc,x_ky_k]))\,\text d\mf y=\int_{H^k/N^k}\int_{N^k}
\mf m_H(C_H([x_1y_1z_1,\dotsc,x_ky_kz_k]))\,\text d\mf z\,\text d\bar{\mf y}\\
&=\int_{H^k/N^k}\int_{N^k}\mf m_N(C_N([x_1y_1z_1,\dotsc,x_ky_kz_k]))\mf m_{G/N}\pr{NC_H([x_1y_1z_1,\dotsc,x_ky_kz_k])/N}\,\text d\mf z\,\text d\bar{\mf y}\\
&\overset{(*)}{\leq}\int_{H^k/N^k}\int_{N^k}\mf m_N(C_N([x_1y_1z_1,\dotsc,x_ky_kz_k]))\mf m_{H/N}\pr{C_{H/N}([\overline{x_1y_1},\dotsc,\overline{x_ky_k}])}\,\text d\mf z\,\text d\bar{\mf y}\\
&=\int_{H^k/N^k}\mf m_{H/N}\pr{C_{H/N}([\overline{x_1y_1},\dotsc,\overline{x_ky_k}])}
\pr{\int_{N^k}\mf m_N(C_N([x_1y_1z_1,\dotsc,x_ky_kz_k]))\,\text d\mf z}\,\text d\bar{\mf y}\\
&=\int_{H^k/N^k}\mf m_{G/N}\pr{C_{G/N}([\overline{x_1y_1},\dotsc,\overline{x_ky_k}])}
\mf{np}(N;x_1y_1,\dotsc,x_ky_k,1)\,\text d\bar{\mf y}\\
&\leq\mf{np}_{k,G}(N)\int_{(G/N)^k}\mf m_{G/N}\pr{C_{G/N}([\overline{x_1}\overline{y_1},\dotsc,\overline{x_k}\overline{y_k}])}\,\text d\bar{\mf y}\\
&\leq\mf{np}_{k,G}(N)\mf{np}_{k,G/N}(H/N).
\end{align*}
\end{footnotesize}
The inequality (*) obtained by the inclusion
\[NC_H([x_1y_1,\dotsc,x_ky_k])/N\subset C_{H/N}([\overline{x_1y_1},\dotsc,\overline{x_ky_k}]).\]
\end{proof}
As a corollary of the above theorem, we state:
\begin{cor}\label{increase wrt Nv2}
Let $N$ be a closed normal subgroup of a compact group $G$. Then $\mf{np}_k(G)\leq\mf{np}_k(G/N)\mf{np}_{k,G}(N)$. In particular, $\mf{cp}(G)\leq\mf{cp}(G/N)\mf{cp}(N)$.
\end{cor}
\section{Proofs of Main Results}
Using Corollaries \ref{increase wrt Nv2} and \ref{3/2}, we prove the following.
\begin{thm}\label{finitelength}
Let $G$ be a compact group and assume that $\mf{np}_k(G)>0$. Then the length of all closed normal series of $G$ which its factors are not nilpotent of class at most $k$ should be bounded.
\end{thm}
\begin{proof}
Let 
\[ 1=G_{r+1}\leq G_{r}\leq\dots\leq G_1\leq G_0=G\]
be a closed normal series such that $G_i/G_{i+1}$, $0\leq i\leq r$, is not nilpotent of class at most $k$. Using the Theorem \ref{increase wrt N} repeatedly, we have 
\begin{align*}
\mf{np}_k(G)&\leq\mf{np}_{k,G_0}(G_r)\mf{np}_k(G_0/G_r)\\
&\leq\mf{np}_{k,G_0}(G_r)\mf{np}_{k,G_0/G_r}(G_{r-1}/G_r)\mf{np}_k(G_0/G_{r-1})\\
&\vdots\\
&\leq\prod_{i=2}^{r+1}\mf{np}_{k,G_0/G_i}(G_{i-1}/G_i).
\end{align*}
The Corollary \ref{3/2} implies that the right hand side of the above inequality is at most 
$\pr{1-\frac{3}{2^{k+1}}}^r$. All in all, we get 
\[ r<\frac{\ln\mf{np}_k(G)}{\ln\pr{1-\frac{3}{2^{k+1}}}}.\]
\end{proof}
\begin{thm}\label{Decresingnet}
Let $G$ be a compact group with $\mf{np}_k(G)>0$. Assume that there exists a decreasing net $(H_\alpha)$ of closed normal subgroups of $G$ such that $\bigcap_\alpha H_\alpha=1$. Then $H_\alpha$ is nilpotent of class at most $k$ for sufficiently large $\alpha$.
\end{thm}
\begin{proof}
If none of $H_\alpha$'s are nilpotent of class at most $k$, then we can construct an infinite closed normal series of $G$ which its factors are not nilpotent of class at most $k$, violates Theorem \ref{finitelength}; indeed, suppose we have constructed a closed normal series, 
\begin{equation}\label{ons}
G=H_0\geq H_{\alpha_1}\geq\dots\geq H_{\alpha_n}
\end{equation} 
such that its factors are not nilpotent of class at most $k$. Since $H_{\alpha_n}$ is not $k$-step nilpotent, there exists $(x_1,\dotsc,x_{k+1})$ in $H_{\alpha_n}^{k+1}$ such that $[x_1,\dotsc,x_{k+1}]\neq1$. By the assumption, we can choose $H_{\alpha_{n+1}}$ such that $[x_1,\dotsc,x_{k+1}]\not\in H_{\alpha_{n+1}}$, so $H_{\alpha_{n+1}}$ can be added to (\ref{ons}).
\end{proof}
Since $\bigcap\set{N: N\,\text{is an open normal subgroup of } G}=1$, for a profinite group $G$, Theorem \ref{Decresingnet} gives another proof for the following well-known fact:
\begin{cor}(\cite[Theorem 1.19.]{MTVV})
Let $G$ be a profinite group with $\mf{np}_k(G)>0$. Then $G$ has a $k$-step nilpotent open subgroup.
\end{cor}
Theorem \ref{finitelength} has also another implication:
\begin{cor}\label{gotofinite}
Let $(G_\alpha)_{\alpha}$ be a collection of compact groups. If $$\mathbf{np}_k\left(\prod_\alpha G_\alpha\right)>0,$$ then $G_\alpha$'s are nilpotent of class at most $k$ for all but finitely many $\alpha$. 
\end{cor}
Now we're ready to show that the class $\mathsf{PA}_k$ contains connected compact groups. A little preparation is needed before proving the later. 
\begin{lem}\label{zariskiclosed} 
Let $G$ be a compact group. Then $G$ contains no nontrivial closed semi simple connected compact Lie subgroup $H$ with {\small $\mf{np}_{k,G}(H)>0$}. 
\end{lem}
\begin{proof}
Suppose $\mathcal N=\mathcal N(H;x_1,\dotsc,x_{k+1})$ has positive measure for some $x_1,\dotsc, x_{k+1}\in G$. Since $\mathcal N$ is Zariski-closed and has positive measure, it should be equal to $G$, so $H$ should be nilpotent by Proposition \ref{nocamn} and hence trivial by semi simplicity.   
\end{proof}
\begin{thm}\label{CCG} 
Let $G$ be a connected compact group with $\mf{np}_k(G)>0$. Then $G$ is an abelian group. 
\end{thm}
\begin{proof}
Note first that simply connected compact groups are semi simple (\cite[Theorem 9.29.]{HM}). Also, Corollary \ref{increase wrt Nv2} implies that $\mf{np}_k(G/\mf Z)>0$, where $\mf Z$ stands for the center of $G$. Therefore by the proof of 
\cite[Theorem 9.24]{HM} and Corollary \ref{gotofinite}, there is a finite family $\{L_1,\dotsc,L_n\}$ of  simple simply connected compact Lie groups such that 
$\frac{G}{\mf Z}\cong\prod_iL_i$, whence implies $\prod_i\mf{np}_k(L_i)>0$. Now the result follows from Lemma \ref{zariskiclosed}. 
\end{proof}
Also compact Lie groups belong to $\mathsf{PA}_k$. Next we use $G^0$ to denote the connected component of a compact group $G$.
\begin{lem}\label{CLG}
Let $G$ be a compact Lie group with $\mf{np}_k(G)>0$. Then $G^0$ is torus. In particular $G$ has an open $k$-step nilpotent subgroup.
\end{lem}
\begin{proof}
Denote by $\mf Z_0$ the center of $G^0$. By Corollary \ref{increase wrt Nv2} and Theorem \ref{increase wrt N}, we obtain
\[ 0<\mf{np}_k(G)\leq\mf{np}_{k,G}(G^0)\leq\mf{np}_{k,G/\mf Z_0}(G^0/\mf Z_0).\]
Now the result obtained with Corollary \ref{zariskiclosed}, because $G^0/\mf Z_0$ is a semi simple connected compact Lie subgroup of $G/\mf Z_0$.The last claim of Lemma follows from the fact that $G^0$ is a closed subgroup of $G$ of finite index (\cite[Theorem 6.36.]{HM}), so it should be open.
\end{proof}
\begin{rem}\label{YCor} 
The above lemma is attributed to Yves de Cornulier (\cite{Ycor}), which is given in response to a question posed by the second author (\cite{msol}). Of course, the above proof is somewhat different from the proof presented in \cite{Ycor}.
\end{rem}
The main result of the Lemma \ref{CLG} is true for arbitrary compact groups. Denote by $\mathfrak{N}(G)$ the set of all closed normal subgroups $N$ of $G$ such that $G/N$ is a Lie group. 
\begin{proof}[Proof of Theorem \ref{main}]
By Corollary \ref{increase wrt Nv2}, $\mf{np}_k(G/N)>0$ for all $N\in\mathfrak{N}(G)$. Lemma \ref{CLG}, now implies that $(G/N)^0=G^0N/N$ (\cite[Lemma 9.18.]{HM}) is an open abelian group and so, $G^0N$ becomes open and normal in $G$. If none  of $G^0N$'s is $k$-step nilpotent, $N\in\mathfrak{N}(G)$, then we can construct a closed normal series of $G$ whose factors are not nilpotent of class at most $k$, violates Theorem \ref{finitelength}; indeed, suppose we have constructed a closed normal series,
\begin{equation}\label{ons}
G\geq G^0N_1\geq\dots\geq G^0N_s
\end{equation} 
such that its factors are not nilpotent of class at most $k$ and $N_1,\dotsc, N_s\in\mathfrak{N}(G)$. Since $G^0N_s$ is not $k$-step nilpotent, there exists $(x_1,\dotsc,x_{k+1})$ in $(G^0N_s)^{k+1}$ such that $[x_1,\dotsc,x_{k+1}]\neq1$. But, $[x_1,\dotsc,x_{k+1}]$ is equal 1 modulus $N_s$, for $G^0N_s/N_s$ is abelian. Therefore $[x_1,\dotsc,x_{k+1}]\in N_s$, and by the fact that $\bigcap\{N: N\in \mathfrak{N}(G)\}=1$ we obtain $N_{s+1}$ in $\mathfrak{N}(G)$, which does not contain $[x_1,\dotsc,x_{k+1}]$.
\end{proof}

The following question is proposed in \cite{AS0} in order to resolve Question \ref{Question}.

\begin{qu}[Question 3.2 of \cite{AS0}] Are there finitely many words $w_{i,j}$ ($1 \leq i \leq n$, $1 \leq j \leq k + 1$) in the free group
on $k + 1$ generators such that if $G$ is a compact group, $(x_1, \dots , x_{k+1}), \mf u = (u_1, \dots , u_{k+1}) \in G^{k+1}$ and 
$[x_1 w_{i,1}(\mf u), \dots , x_{k+1} w_{i,k+1}(\mf u)] = 1$ for all $i \in \{1, \dots , n\}$ then $[u_1, \dots , u_{k+1}] = 1$?
\end{qu}

To settle Question \ref{Question}, we will prove that using Theorem \ref{main}, we may study the following slightly simpler question on finite groups.
  
\begin{qu}\label{qusnew} Are there finitely many words $w_{i,j}$ ($1 \leq i \leq n$, $1 \leq j \leq k + 1$) in the free group
	on $k + 1$ generators such that if $G$ is a finite group, $(x_1, \dots , x_{k+1}), \mf u = (u_1, \dots , u_{k+1}) \in G^{k+1}$, $\langle u_1, \dots , u_{k+1} \rangle$ is abelian-by-nilpotent and 	$[x_1 w_{i,1}(\mf u), \dots , x_{k+1} w_{i,k+1}(\mf u)] = 1$ for all $i \in \{1, \dots , n\}$ then $[u_1, \dots , u_{k+1}] = 1$?
\end{qu}

\begin{thm}
If the answer of Question \ref{qusnew} is positive, the answer of Question \ref{Question} is also affirmative.
\end{thm}
\begin{proof}
First note that by Theorem \ref{main} together with Hall Theorem, \cite[15.4.1]{Rob}, all finitely generated subgroups of $G$ are residually finite.\par 
Let $n$ be the positive integer in Question \ref{qusnew}. By \cite[Theorem 3.1]{AS0} there exists an open subset $V$ of $G$ containing $1$ such that 
\begin{align}\label{cap}
\mathbf{m}\pr{\bigcap_{i=1}^n \mathcal{N}_k(G)\mf y_j}>0,
\end{align}
where $\mf y_i\in V^{k+1}$, $i=1,\dotsc,n$. By continuity of word maps, we can find an open neighborhood $U$ of $1$, such that $w_{i,j}(\mf u)\in V$ for 
$$\mf u=(u_1,\dotsc,u_{k+1})\in U^{k+1},$$ and $1 \leq i \leq n$ , $1 \leq j \leq k + 1$. Now, for $\mf u=(u_1,\dotsc,u_{k+1})\in U^{k+1}$, (\ref{cap}) ensures the existence of $x_1, \dots , x_{k+1}\in G$, such that 
\begin{align}
\label{EQ}
[x_1 w_{i,1}(\mf u), \dots , x_{k+1} w_{i,k+1}(\mf u)] = 1,\quad (1 \leq i \leq n).
\end{align}
But $\valu{x_i, u_i: i=1,\dotsc, k+1}$ is residually finite, so by passing to its finite quotients and the correctness of the Question \ref{qusnew}, we get from (\ref{EQ}) that $[u_1, \dots , u_{k+1}] = 1$, whence the open subgroup $\valu{U}$ is a $k$-step open subgroup of $G$.
\end{proof}

\end{document}